\documentclass[a4paper,10pt]{amsart}
\usepackage{amsmath}
\usepackage{cases}

\UseRawInputEncoding

\usepackage{amsfonts}
\usepackage[colorlinks,linkcolor=blue,citecolor=blue]{hyperref}
\usepackage{latexsym, amssymb, amsmath, amsthm, bbm}
\usepackage[all]{xy}
\usepackage{pgfplots}
\usepackage{enumerate}

\DeclareSymbolFont{EulerExtension}{U}{euex}{m}{n}
\DeclareMathSymbol{\euintop}{\mathop} {EulerExtension}{"52}
\DeclareMathSymbol{\euointop}{\mathop} {EulerExtension}{"48}

\allowdisplaybreaks[4]

\def \id{\operatorname{id}}
\def \Ker{\operatorname{Ker}}
\def \ord{\operatorname{ord}}

\def \k{\mathbbm{k}}

\def \dim{\operatorname{dim}}
\def \gr{\operatorname{gr}}
\def \Hom{\operatorname{Hom}}

\def \ord{\operatorname{ord}}

\def \Ker{\operatorname{Ker}}

\def \End{\operatorname{End}}
\def \qexp{\operatorname{qexp}}
\def \lcm{\operatorname{lcm}}

\numberwithin{equation}{section}

\newtheorem{theorem}{Theorem}[section]
\newtheorem{lemma}[theorem]{Lemma}
\newtheorem{proposition}[theorem]{Proposition}
\newtheorem{corollary}[theorem]{Corollary}
\newtheorem{definition}[theorem]{Definition}

\newtheorem{remark}[theorem]{Remark}

\newtheorem{notation}[theorem]{Notation}

\begin{document}
\title{On the antipode of Hopf algebras with the dual Chevalley property}
\thanks{$^\dag$Supported by NSFC 11722016.}


\author{Kangqiao Li and Gongxiang Liu}
\address{Department of Mathematics, Nanjing University, Nanjing 210093, China}
\email{kqli@nju.edu.cn}
\email{gxliu@nju.edu.cn}

\date{}
\maketitle

\begin{abstract}
In this paper, we study the antipode of a finite-dimensional Hopf algebra $H$ with the dual Chevalley property and obtain an annihilation polynomial for its antipode $S$.
The annihilation polynomial is determined by the exponent $N$ of the coradical and the Loewy length. In particular, the order of $S^2$ divides $N$ in characteristic $0$. Moreover, we get two characterizations of the quasi-exponent.
\end{abstract}

\section{Introduction}

Let $H$ be a finite-dimensional Hopf algebra over a field $\k$ with the antipode $S$, and denote the composition order of $S^2$ by $\ord(S^2)$. The order or annihilation polynomials of $S^2$ has been studied for more than 40 years. The first most general result is given by Radford \cite{Radford 1976} in 1976, which states that $\ord(S^2)$ is always finite. Then the people want to find an explicit bound for $\ord(S^2)$. Actually, the people made the progress at least in the following two different cases for $H$: semisimple case and pointed case.

As for the first case when $H$ is semisimple, Kaplansky conjectured in \cite{Kaplansky 1975} that semisimple Hopf algebras are all involutory (that is, $\ord(S^2)=1$), which is well-known as the Kaplansky's fifth conjecture and still open in small positive characteristic. Moreover, when $H$ is cosemisimple in addition, a positive answer was given by Etingof and Gelaki \cite{E-G 1998}. The other case when $H$ is pointed was once studied by Taft and Wilson \cite{T-W 1974} in 1974. They obtained the following annihilation polynomial:
\begin{equation}\label{28}
(S^{2N}-\id)^{L-1}=0,
\end{equation}
where $N$ is the exponent of the coradical and $L$ denotes the Loewy length. This formula implies directly that $\ord(S^2)\mid N$ in characteristic $0$ (\cite{E-G 2002}) and $\ord(S^2)\mid Np^M$ in characteristic $p>0$ for some positive integer $M$ (\cite{T-W 1974}).

Our first goal in this paper is to show that Formula (\ref{28}) holds as well when $H$ has the dual Chevalley property, which provides the desired annihilator polynomial for the antipode. The main tools for the proof are so-called coradical orthonormal idempotents in $H^\ast$ and (multiplicative and primitive) matrices over $H$. Multiplicative and primitive matrices could be regarded as analogous concepts to grouplike and primitive elements, respectively. Such concepts are studied in \cite{Li-Zhu 2019}, and applied to study properties of non-pointed coalgebras or Hopf algebras. For example, if the coradical filtration is denoted by $\{H_n\}_{n\geq 0}$, one could decompose an arbitrary element in $H_1$ into a sum of entries coming from multiplicative and primitive matrices.

The idea for setting the annihilator polynomial can be described as follows. With the help of our developed tools, we find that the action of $S^2$ on primitive matrices behaves similar to a conjugate action by a multiplicative matrix at first. Then along the similar line as the pointed case, we show that the action of $S^{2N}$ on $H_1$ is exactly the identity map. At last, by a lemma in \cite{T-W 1974} which gives us an inductive way from $H_n$ to $H_{n+1}$, where $\{H_n\}_{n\geq 0}$ is the coradical filtration of $H$, we get our desired formula.

Afterwards we provide some other results followed. The first one is that the order of $S^2$ divides $\exp(H_0)$ in characteristic $0$. This implies our second application that the quasi-exponent of $H$ is exactly $\exp(H_0)$. Finally, by specific calculations, we show that how the quasi-exponent of the semi-direct product Hopf algebra $H\rtimes \k \langle S^2\rangle$ is determined by that of $H$.

The organization of this paper is as follows: In Section \ref{b}, we recall definitions and properties of tools we need, including the exponent, coradical orthonormal idempotents, as well as multiplicative and primitive matrices mentioned above. Our first result on the annihilator polynomial is stated and proved in Section \ref{c}. A direct corollary that $\ord(S^2)\mid\exp(H_0)$ in characteristic $0$ is also obtained in this section. Some properties on quasi-exponents are then given as applications in the last section.

\section{Preliminaries}\label{b}

We recall the most needed knowledge, including the definitions and some properties of multiplicative and primitive matrices, in this section. Let $\k$ be a field throughout this paper, and the tensor product over $\k$ is always denoted simply by $\otimes$. For a coalgebra $(H,\Delta,\varepsilon)$ over a field $\k$, Sweedler notation $\Delta(h)=\sum h_{(1)}\otimes h_{(2)}$ for $h\in H$ is always used.

\subsection{Dual Chevalley Property}

A finite-dimensional Hopf algebra $H$ is said to have the \textit{dual Chevalley property}, if its coradical $H_0$ is a Hopf subalgebra (or equivalently, $H_0$ is a subalgebra of $H$). A well-known result about the dual Chevalley property is the following lemma (see e.g. \cite[Lemma 5.2.8]{Montgomery 1993}).

\begin{lemma}\label{14}
Let $H$ be a Hopf algebra $H$ with the coradical filtration $\{H_n\}_{n\geq 0}$. Then the followings are equivalent:
\begin{itemize}
  \item[(1)] $H_0$ is a Hopf subalgebra of $H$;
  \item[(2)] $\{H_n\}_{n\geq 0}$ is a Hopf algebra filtration.
\end{itemize}
\end{lemma}

\subsection{Exponent}

Let $(H,m,u,\Delta,\varepsilon)$ be a Hopf algebra with antipode $S$ over a field $\k$. For convenience, we define following $\k$-linear maps for any positive integer $n$:
\begin{eqnarray*}
m_n : & H^{\otimes n}\rightarrow H, & h_1\otimes h_2\otimes \cdots \otimes h_n \mapsto h_1 h_2\cdots h_n ;\\
\Delta_n : & H\rightarrow H^{\otimes n}, & h\mapsto \sum h_{(1)}\otimes h_{(2)}\otimes \cdots \otimes h_{(n)}.
\end{eqnarray*}
When $S$ is bijective, the notion of the \textit{exponent} of $H$ introduced in \cite{E-G 1999} by Etingof and Gelaki is defined as
$$\exp(H):=\min\{n\geq 1\mid m_n\circ (\id\otimes S^{-2}\otimes\cdots\otimes S^{-2n+2})\circ \Delta_n=u\circ\varepsilon \}$$
with convention $\min \varnothing=\infty$. One of their most crucial ways to study the exponent is the following identification \cite[Theorem 2.5(2)]{E-G 1999}:
$$\exp(H) ~\text{equals to the multiplication order of}~u_{D(H)},$$
where $u_{D(H)}$ is the Drinfeld element of the Drinfeld double $D(H)$ (\cite{Drinfeld 1986}). We remark that there is also another notion of ``exponent" introduced by Kashina \cite{Kashina 1999,Kashina 2000} (see also Remark \ref{r4.10}).

It is known that $D(H)$ is also a Hopf algebra, whose antipode is denoted by $S_{D(H)}$. Moreover, $S_{D(H)}{}^2$ is in fact an inner automorphism determined by $u_{D(H)}$ on $D(H)$. Thus $S_{D(H)}{}^{2\exp(H)}$ becomes the identity map on $D(H)$, as long as $\exp(H)<\infty$. Restricting this map onto the Hopf subalgebra $H\cong \varepsilon\bowtie H\subseteq D(H)$, we obtain the following fact immediately:

\begin{corollary}\label{31}
Let $H$ be a finite-dimensional Hopf algebra with antipode $S$. If $\exp(H)<\infty$, then $S^{2\exp(H)}=\id$ \emph{(}the identity map on $H$\emph{)}.
\end{corollary}

There are two theorems \cite[Theorem 4.3]{E-G 1999} and \cite[Theorem 4.10]{E-G 1999} describing the finiteness of the exponent. We state them below:
\begin{lemma}
Let $H$ be a finite-dimensional Hopf algebra over $\k$.
\begin{itemize}
  \item[(1)] If $H$ is semisimple and cosemisimple, then $\exp(H)$ is finite and divides $\dim(H)^3$;
  \item[(2)] If ${\rm char}~\k>0$, then $\exp(H)<\infty$.
\end{itemize}
\end{lemma}
As a conclusion of the theorems above, it is easy to check that the coradical of a finite-dimensional Hopf algebra with the dual Chevalley property always has finite exponent:

\begin{corollary}\label{1}
Let $H$ be a finite-dimensional Hopf algebra with the dual Chevalley property over $\k$. Then $\exp(H_0)<\infty$.
\end{corollary}

\begin{proof}
If ${\rm char}~\k>0$, then this is a direct consequence of the (2) of the above lemma. If ${\rm char}~\k=0$,  the cosemisimple Hopf algebra $H_0$ is also semisimple now by \cite[Theorem 3.3]{L-R 1988}. Then (1) of the above lemma is applied.
\end{proof}

\subsection{Coradical Orthonormal Idempotents}\label{29}

For any coalgebra $H$, its dual algebra with the convolution product is denoted by $H^\ast$. Now we refer a certain kind of family of idempotents in $H^\ast$ introduced by Radford \cite{Radford 1978}, which are called \emph{coradical orthonormal idempotents} in this paper. To introduce them, let $\mathcal{S}$ be the set of simple subcoalgebras of $H$ and the classical Kronecker delta is denoted by $\delta$.

\begin{definition}
Let $H$ be a coalgebra. A family of coradical orthonormal idempotents of $H^\ast$ is a family of non-zero elements $\{e_C\}_{C\in\mathcal{S}}$ in $H^\ast$ satisfying following conditions:
\begin{itemize}
  \item[(1)] $ e_C e_D=\delta_{C,D}e_C$ for $C,D\in \mathcal{S}$;
  \item[(2)] $\sum\limits_{C\in\mathcal{S}} e_C=\varepsilon$ on $H$ (distinguished condition);
  \item[(3)] $e_C|_D=\delta_{C,D}\varepsilon|_D$ for $C,D\in \mathcal{S}$.
\end{itemize}
\end{definition}

The existence of (a family of) coradical orthonormal idempotents in $H^\ast$ for any coalgebra $H$ is affirmed in Radford \cite[Lemma 2]{Radford 1978} or \cite[Corollary 3.5.15]{Radford 2012}, according to properties of injective comodules. It is always assumed that $\{e_C\}_{C\in\mathcal{S}}$ is a given family of coradical orthonormal idempotents in $H^\ast$ for the remaining of this paper.

We remark that when $H$ is pointed, there is another way to construct coradical orthonormal idempotents from \cite[Theorem 5.4.2]{Montgomery 1993} for example, and some convenient notations are used there. Similar notations for $\{e_C\}_{C\in\mathcal{S}}$ will be used in this paper too:
$${^C}h=h\leftharpoonup e_C,h^D=e_D\rightharpoonup h,{^C}h^D=e_D\rightharpoonup h\leftharpoonup e_C,\ ~h\in H,~C,D \in \mathcal{S},$$
where $\leftharpoonup$ and $\rightharpoonup$ are hit actions of $H^\ast$ on $H$. Specially if $C={\k} g$ is pointed, then we also denote ${}^gh:={}^{C}h$, $h^g:=h^{C}$, and  $V^{C}:=e_C\rightharpoonup V$, etc..

Some direct properties, which can be found in \cite[Proposition 2.2]{Li-Zhu 2019}, are listed as follows:

\begin{proposition}\label{3}
Let $H$ be a coalgebra. Then for all $C,D\in\mathcal{S}$, we have
\begin{itemize}
\item[(1)] ${}^C H_0{}^D=\delta_{C,D}C$;
\item[(2)] ${}^C H_1{}^D \subseteq \Delta^{-1}(C\otimes {}^C H_1{}^D + {}^C H_1{}^D \otimes D)$;
\item[(3)] ${}^C H{}^D \subseteq \Ker(\varepsilon)$ if $C\neq D$;
\item[(4)] Suppose $V\subseteq H$ is a ${\k}$-subspace. We have following direct-sum decomposition:
      \begin{itemize}
        \item[(i)] $V=\bigoplus\limits_{C\in\mathcal{S}}{}^C V$ if $V$ is a left coideal;
        \item[(ii)] $V=\bigoplus\limits_{D\in\mathcal{S}} V^D$ if $V$ is a right coideal;
        \item[(iii)] $V=\bigoplus\limits_{C,D\in\mathcal{S}} {}^C V^D$ if $V$ is a subcoalgebra.
      \end{itemize}
\end{itemize}
\end{proposition}

\subsection{Multiplicative and Primitive Matrices}
 The content of this subsection could be found in \cite[Section 3]{Li-Zhu 2019}. For positive integers $r$ and $s$, we use $\mathcal{M}_{r\times s}(V)$ to denote the set of all $r\times s$ matrices over a vector space $V$. If $r=s$, we write $\mathcal{M}_r(V)=\mathcal{M}_{r\times r}(V)$.

\begin{notation}
Let $V,W$ be ${\k}$-vector spaces.
\begin{itemize}
  \item[(1)] Define a bilinear map
        \begin{eqnarray*}
        \widetilde{\otimes}: &\mathcal{M}_{r\times s}(V)\otimes\mathcal{M}_{s\times t}(W)\rightarrow \mathcal{M}_{r\times t}(V\otimes W), \\
        & (v_{ij})\otimes(w_{kl}) \mapsto \left( \sum\limits_{k=1}^s v_{ik} \otimes w_{kj} \right).
        \end{eqnarray*}
        Note that $\widetilde{\otimes}\circ(\widetilde{\otimes}\otimes \id)=\widetilde{\otimes}\circ(\id\otimes\widetilde{\otimes})$ holds on
        $\mathcal{M}_{r\times s}(U)\otimes\mathcal{M}_{s\times t}(V)\otimes\mathcal{M}_{t\times u}(W)$.
  \item[(2)] Let $f\in  {\Hom}_{{\k}}(V,W)$, we always use the same symbol $f$ to denote the linear map
        $f:\mathcal{M}_{r\times s}(V)\rightarrow \mathcal{M}_{r\times s}(W),(v_{ij})\mapsto \left(  f(v_{ij}) \right)$.
\end{itemize}
\end{notation}

\begin{remark}
\emph{With such notations, it is direct that when $(H,m,u)$ is an algebra, the linear map $m\circ\widetilde{\otimes}$ on $\mathcal{M}_{r\times s}(H)\otimes \mathcal{M}_{s\times t}(H)$ is exactly the matrix multiplication $\mathcal{M}_{r\times s}(H)\otimes \mathcal{M}_{s\times t}(H)\to \mathcal{M}_{r\times t}(H)$.}
\end{remark}

A ``multiplicative matrix'' over a coalgebra, once introduced in \cite[Section 2.6]{Manin 1988} for quantum group constructions, has similar properties to a grouplike element.

\begin{definition}
Let $(H,\Delta,\varepsilon)$ be a coalgebra and $r$ be a positive integer. Let $I_r$ denote the unit matrix of order $r$ over ${\k}$.
\begin{itemize}
  \item[(1)] A matrix $\mathcal{G}\in \mathcal{M}_r(H)$ is called a multiplicative matrix over $H$ if $\Delta(\mathcal{G})=\mathcal{G}~\widetilde{\otimes}~\mathcal{G}$ and $\varepsilon(\mathcal{G})=I_r$.
  \item[(2)] For any $C\in \mathcal{S}$, a multiplicative matrix $\mathcal{C}$ is called a basic multiplicative matrix of $C$, if all the entries of $\mathcal{C}$ form a linear basis of $C$.
\end{itemize}
\end{definition}

\begin{remark}
\emph{It is well-known that every simple coalgebra over an algebraically closed field has a basic multiplicative matrix.}
\end{remark}

Moreover, if $H$ is a bialgebra, then the $n$th Sweedler power $[n]:h\mapsto\sum h_{(1)}h_{(2)}\cdots h_{(n)}$ of a multiplicative matrix $\mathcal{G}=(g_{ij})$ over $H$ equals to the $n$th (multiplication) power of $\mathcal{G}$ (see \cite[Corollary 3]{K-S-Z 2006} or \cite[Proposition 4.2]{Li-Zhu 2019}).

\begin{lemma}
Let $(H,m,u,\Delta,\varepsilon)$ be a ${\k}$-bialgebra. Let $\mathcal{G}$ be an $r\times r$ multiplicative matrix over $H$. Then
\begin{itemize}
\item[(1)] $\mathcal{G}^{[n]}:=(g_{ij}{}^{[n]})=\mathcal{G}^n$ for any positive integer $n$.
\item[(2)] If $H$ is a Hopf algebra with antipode $S$, then $S(\mathcal{G})\mathcal{G}=\mathcal{G}S(\mathcal{G})=I_r$.
\end{itemize}
\end{lemma}

Primitive elements play an important role in the study of pointed coalgebras, and ``primitive matrices'' play a similar role in non-pointed case.

\begin{definition}
Let $H$ be a coalgebra, and let $\mathcal{C}_{r\times r},\mathcal{D}_{s\times s}$ be two multiplicative matrices.
A matrix $\mathcal{W}\in\mathcal{M}_{r\times s}(H)$ is called a $\left(\mathcal{C},\mathcal{D}\right)$-primitive matrix if $\Delta(\mathcal{W})=\mathcal{C}~\widetilde{\otimes}~\mathcal{W}+\mathcal{W}~\widetilde{\otimes}~\mathcal{D}$.
\end{definition}

\begin{remark}
\emph{It is easy to show that $\varepsilon(\mathcal{W})=0$ for any primitive matrix $\mathcal{W}$.}
\end{remark}

Using the method of coradical orthonormal idempotents, we could express any element in $H_1$ as a sum of entries in multiplicative and primitive matrices. For any subspace $V\subseteq H$, we denote $V\cap \Ker(\varepsilon)$ by $V^+$. The following lemma is important for us (see \cite[Theorem 3.1]{Li-Zhu 2019}).

\begin{lemma}\label{4}
Assume that $\k$ is algebraically closed. Let $H$ be a coalgebra, $C,D$ be simple subcoalgebras of $H$ and $\mathcal{C}=(c_{i^\prime i})_{r\times r},\mathcal{D}=( d_{jj^\prime })_{s\times s}$ be respectively basic multiplicative matrices for $C$ and $D$. Then
\begin{itemize}
\item[(1)] If $C\neq D$, then for any $w\in {}^C H_1{}^D$, there exist $rs$-number of  $\left(\mathcal{C},\mathcal{D}\right)$-primitive matrices
        $$
        \mathcal{W}^{(i^\prime ,j^\prime )}= \left( w_{ij}^{(i^\prime ,j^\prime )} \right)_{r\times s}~~
        (1\leq i^\prime \leq r,1\leq j^\prime \leq s),
        $$
        such that $w=\sum\limits_{i=1}^r \sum\limits_{j=1}^s w_{ij}^{(i,j)}$;
\item[(2)] If $C=D$ and we choose that $\mathcal{D}=\mathcal{C}$, then for any $w\in {}^C H_1{}^C$, there exist $r^2$-number of $\left(\mathcal{C},\mathcal{C}\right)$-primitive matrices
        $$
        \mathcal{W}^{(i^\prime ,j^\prime )}= \left( w_{ij}^{(i^\prime ,j^\prime )} \right)_{r\times s}~~
        (1\leq i^\prime,j^\prime \leq r)
        $$
        such that $w- \sum\limits_{i=1}^r \sum\limits_{j=1}^s w_{ij}^{(i,j)}\in C$.
\end{itemize}
\end{lemma}

\section{An annihilation polynomial for antipode}\label{c}

Let $H$ be a finite-dimensional Hopf algebra over $\k$ with the dual Chevalley property and $H_0$ be its coradical. We denote its Loewy length by ${\rm Lw}(H)$. In this section, our aim  is to prove the following annihilation polynomial for $S^2\in\End_\k(H)$:

\begin{theorem}\label{2}
Let $H$ be a finite-dimensional Hopf algebra with the dual Chevalley property. Denote $N:=\exp(H_0)<\infty$ and $L:={\rm Lw}(H)$. Then $$(S^{2N}-\id)^{L-1}=0$$ holds on $H$.
\end{theorem}

\begin{remark}
\emph{For a finite-dimensional pointed Hopf algebra, the same annihilation polynomial was established by Taft and Wilson \cite[Theorem 5]{T-W 1974} in 1974. It could be implied by Theorem \ref{2}, since finite-dimensional pointed Hopf algebras clearly have the dual Chevalley property.}
\end{remark}

Before the proof, an immediate but meaningful conclusion on the order of the antipode should be noted as follows, which implies \cite[Theorem 4.4]{E-G 2002} as well as \cite[Corollary 6]{T-W 1974} for the pointed case. We denote the composition order of $S^2$ by $\ord(S^2)$.

\begin{corollary}\label{11}
Let $H$, $N$ and $L$ be as in Theorem \ref{2}. Then
\begin{itemize}
  \item[(1)] If ${\rm char}~\k=0$, then $\ord(S^2)\mid N$;
  \item[(2)] If ${\rm char}~\k=p>0$, then $\ord(S^2)\mid \gcd(Np^M,\exp(H))$, where $M$ is a natural number satisfying $p^M\geq L-1$.
\end{itemize}
\end{corollary}

\begin{proof}
\begin{itemize}
  \item[(1)] The order of $S$ is finite, for $H$ is finite-dimensional \cite[Theorem 1]{Radford 1976}. Then $S^{2N}$ is semisimple in characteristic $0$, but unipotent. It follows that $S^{2N}=\id$.
  \item[(2)] It is evident that $S^{2Np^M}-\id=(S^{2N}-\id)^{p^M}=0$. On the other hand, Lemma \ref{31} implies that $\ord{(S^2)}\mid\exp(H)$.
\end{itemize}
\end{proof}

We divide the proof of Theorem \ref{2} into several steps which occupy the following subsections.

\subsection{Antipode on Primitive Matrices}
First of all, we show how algebra anti-endomorphisms act on products of matrices over an (associative) algebra. For a matrix $\mathcal{A}=(a_{ij})_{r\times s}$ over an algebra, we denote the transpose of $\mathcal{A}$ by
$\mathcal{A}^{\rm T}:=(a_{ji})_{s\times r}$.

\begin{lemma}\label{5}
Let $H$ be an associative algebra with an algebra anti-endomorphism $S$. For any matrices $\mathcal{A}_1,\mathcal{A}_2,\cdots,\mathcal{A}_n$ over $H$, we have
$$S(\mathcal{A}_1\mathcal{A}_2\cdots\mathcal{A}_n)^{\rm T}=S(\mathcal{A}_n)^{\rm T}S(\mathcal{A}_{n-1})^{\rm T}\cdots S(\mathcal{A}_1)^{\rm T}$$
as long as the product $\mathcal{A}_1\mathcal{A}_2\cdots\mathcal{A}_n$ is well-defined.
\end{lemma}

\begin{proof}
The equation holds due to direct calculations.
\end{proof}

From now on, suppose that $H$ is a finite-dimensional Hopf algebra with the antipode $S$. With the help of the lemma above, we could in fact compute the image of a certain kind of primitive matrices under $S^{2n}$ for each positive integer $n$.

\begin{lemma}\label{6}
Let $\mathcal{C}=(c_{ij})_{r\times r}$ be a multiplicative matrix, and let $\mathcal{X}=(x_1,x_2,\cdots,x_r)^{\rm T}$ be a $(\mathcal{C},1)$-primitive matrix over $H$.
\begin{itemize}
  \item[(1)] $S(\mathcal{X})=-S(\mathcal{C})\mathcal{X}$;
  \item[(2)] $S^2(\mathcal{X})=((S(\mathcal{C})\mathcal{X})^{\rm T}S^2(\mathcal{C})^{\rm T})^{\rm T}$.
        In other words, $S^2(x_{i})=\sum\limits_{k_1,k_2=1}^r S(c_{k_2 k_1})x_{k_1}S^2(c_{i k_2})$ for each $1\leq i\leq r$;
  \item[(3)] For any positive integer $n$,
           $$
           S^{2n}(\mathcal{X})=
           [[S^{2n-1}(\mathcal{C})\cdots((S(\mathcal{C})\mathcal{X})^{\rm T}S^2(\mathcal{C})^{\rm T})^{\rm T}\cdots]^{\rm T}S^{2n}(\mathcal{C})^{\rm T}]^{\rm T}.
           $$
        Specifically, the $(i,1)$-entry of $S^{2n}(\mathcal{X})$ is
           \begin{eqnarray*}
           S^{2n}(x_{i})
           &=& \sum\limits_{k_1,k_2,\cdots,k_{2n}=1}^r S\left[c_{k_2k_1} S^2(c_{k_4k_3}) \cdots S^{2n-4}(c_{k_{2n-2}k_{2n-3}})
                                                  S^{2n-2}(c_{k_{2n}k_{2n-1}})\right] \\
           & & x_{k_1}S^2\left[ c_{k_3k_2} S^2(c_{k_5k_4})\cdots S^{2n-4}(c_{k_{2n-1}k_{2n-2}}) S^{2n-2}(c_{ik_{2n}})\right].
           \end{eqnarray*}
\end{itemize}
\end{lemma}

\begin{proof}
\begin{itemize}
  \item[(1)] The definition of $(\mathcal{C},1)$-primitive matrices means that
        $\Delta(\mathcal{X})=\mathcal{C}~\widetilde{\otimes}~\mathcal{X}+\mathcal{X}~\widetilde{\otimes}~1$.
        We map $m\circ(S\otimes\id)$ onto this equation and obtain
        $$0=\varepsilon(\mathcal{X})=S(\mathcal{C})\mathcal{X}+S(\mathcal{X}),$$
        which follows $S(\mathcal{X})=-S(\mathcal{C})\mathcal{X}$ immediately.
  \item[(2)] According to (1) and lemma \ref{5}, it is direct that
        \begin{eqnarray*}
          S^2(\mathcal{X})
          &=& S(-S(\mathcal{C})\mathcal{X})~=~-(S(\mathcal{X})^{\rm T}S^2(\mathcal{C})^{\rm T})^{\rm T}
          ~=~ -((-S(\mathcal{C})\mathcal{X})^{\rm T} S^2(\mathcal{C})^{\rm T})^{\rm T} \\
          &=& ((S(\mathcal{C})\mathcal{X})^{\rm T} S^2(\mathcal{C})^{\rm T})^{\rm T}.
        \end{eqnarray*}
  \item[(3)] We prove it by induction on $n$. Assume that the equation holds for any multiplicative matrix $\mathcal{C}$ and $(\mathcal{C},1)$-primitive matrix $\mathcal{X}$ in case $n-1$. Note that $S^2(\mathcal{C})$ is multiplicative and $S^2(\mathcal{X})$ is $(S^2(\mathcal{C}),1)$-primitive, because $S^2$ is a coalgebra endomorphism (as well as an isomorphism) on $H$. Then
        \begin{eqnarray*}
        & & S^{2n}(x_{i}) ~=~ S^{2n-2}(S^2(x_{i})) \\
        &=& \sum_{k_3,k_4,\cdots,k_{2n}=1}^r S\left[S^2(c_{k_4k_3})\cdots S^{2n-4}(c_{k_{2n-2}k_{2n-3}}) S^{2n-2}(c_{k_{2n}k_{2n-1}})\right] \\
        & &       S^2(x_{k_3}) S^2\left[S^2(c_{k_5k_4})\cdots S^{2n-4}(c_{k_{2n-1}k_{2n-2}}) S^{2n-2}(c_{ik_{2n}})\right] \\
        &=& \sum_{k_3,k_4,\cdots,k_{2n}=1}^r S\left[S^2(c_{k_4k_3})\cdots S^{2n-4}(c_{k_{2n-2}k_{2n-3}}) S^{2n-2}(c_{k_{2n}k_{2n-1}})\right] \\
        & &       \left(\sum_{k_1,k_2=1}^r S(c_{k_2 k_1})x_{k_1}S^2(c_{k_3 k_2})\right)
                  S^2\left[S^2(c_{k_5k_4})\cdots S^{2n-4}(c_{k_{2n-1}k_{2n-2}}) S^{2n-2}(c_{ik_{2n}})\right] \\
        &=& \sum_{k_1,k_2,k_3,k_4,\cdots,k_{2n}=1}^r
                  S\left[S^2(c_{k_4k_3})\cdots S^{2n-4}(c_{k_{2n-2}k_{2n-3}}) S^{2n-2}(c_{k_{2n}k_{2n-1}})\right] \\
        & &       S(c_{k_2 k_1}) x_{k_1} S^2(c_{k_3 k_2}) S^2\left[S^2(c_{k_5k_4})\cdots S^{2n-4}(c_{k_{2n-1}k_{2n-2}}) S^{2n-2}(c_{ik_{2n}})\right] \\
        &=& \sum_{k_1,k_2,k_3,k_4,\cdots,k_{2n}=1}^r
                  S\left[c_{k_2 k_1} S^2(c_{k_4k_3})\cdots S^{2n-4}(c_{k_{2n-2}k_{2n-3}}) S^{2n-2}(c_{k_{2n}k_{2n-1}})\right] \\
        & &       x_{k_1}S^2\left[c_{k_3 k_2}S^2(c_{k_5k_4})\cdots S^{2n-4}(c_{k_{2n-1}k_{2n-2}}) S^{2n-2}(c_{ik_{2n}})\right],
        \end{eqnarray*}
        which is exactly the required equation in case $n$.
\end{itemize}
\end{proof}

It is suggested in Lemma \ref{6} that the mapping $S^2$ on $(\mathcal{C},1)$-primitive matrices is somehow similar to a ``conjugate action by by $\mathcal{C}$''. We show next that such an action has finite order when $H$ has finite exponent.

\begin{lemma}\label{7}
Let $\mathcal{C}=(c_{ij})_{r\times r}$ be a basic multiplicative matrix of a simple subcoalgebra $C\in\mathcal{S}$ of $H$. Assume that $N:=\exp(H)<\infty$. Then for $ 1\leq i,j\leq r$,
\begin{eqnarray*}
 &&\sum_{k_2,k_3,\cdots,k_{2N}=1}^r c_{k_2 j}S^2(c_{k_4k_3})\cdots S^{2N-4}(c_{k_{2N-2}k_{2N-3}})S^{2N-2}(c_{k_{2N}k_{2N-1}}) \\
 && \ \ \ \ \ \ \ \ \ \ \otimes~ c_{k_3 k_2}S^2(c_{k_5k_4})\cdots S^{2N-4}(c_{k_{2N-1}k_{2N-2}})S^{2N-2}(c_{ik_{2N}}) \\
 && =\delta_{ij}(1\otimes 1),
\end{eqnarray*}
which is the $(i,j)$-entry of the identity matrix $\mathcal{I}_r~\widetilde{\otimes}~\mathcal{I}_r\in\mathcal{M}_r(H\otimes H)$.
\end{lemma}

\begin{proof}
Firstly it is known that $\exp(H^{\ast\rm op})=\exp(H)=N$ according to \cite[Proposition 2.2(4) and Corollary 2.6]{E-G 1999}. Here the antipode of dual Hopf algebra $H^{\ast}$ is also denoted as $S$, and then the antipode of $H^{\ast\rm op}$ is actually $S^{-1}$.

Let us prove the equation by taking values of the function $f\otimes g\in H^\ast\otimes H^\ast$ for arbitrary $f,g\in H^\ast$.
Note that
$$
\Delta_{2N}(c_{ij})=\sum_{k_2,k_3,\cdots,k_{2N}=1}^r
c_{ik_{2N}}\otimes c_{k_{2N}k_{2N-1}}\otimes\cdots\otimes c_{k_3k_2}\otimes c_{k_2j}.
$$
The value of $f\otimes g$ on the left side of the required equation is then
\begin{eqnarray*}
 &&\sum_{k_2,k_3,\cdots,k_{2N}=1}^r \langle f,c_{k_2 j}S^2(c_{k_4k_3})\cdots S^{2N-4}(c_{k_{2N-2}k_{2N-3}})S^{2N-2}(c_{k_{2N}k_{2N-1}})\rangle \\
 && \ \ \ \ \ \ \ \ \langle g, c_{k_3 k_2}S^2(c_{k_5k_4})\cdots S^{2N-4}(c_{k_{2N-1}k_{2N-2}})S^{2N-2}(c_{ik_{2N}})\rangle \\
 &&= \sum_{k_2,k_3,\cdots,k_{2N}=1}^r \langle g_{(N)},S^{2N-2}(c_{ik_{2N}})\rangle \langle f_{(N)},S^{2N-2}(c_{k_{2N}k_{2N-1}})\rangle\\
 &&\ \ \ \ \ \ \ \ \langle g_{(N-1)},S^{2N-4}(c_{k_{2N-1}k_{2N-2}})\rangle
                                 \langle f_{(N-1)},S^{2N-4}(c_{k_{2N-2}k_{2N-3}})\rangle\\
 &&\ \ \ \ \ \ \ \ \cdots \langle g_{(2)},S^2(c_{k_5k_4}) \rangle \langle f_{(2)},S^2(c_{k_4k_3}) \rangle
                                      \langle g_{(1)},c_{k_3k_2} \rangle \langle f_{(1)},c_{k_2j} \rangle \\
 &&= \langle\sum S^{2N-2}(g_{(N)})S^{2N-2}(f_{(N)})S^{2N-4}(g_{(N-1)})S^{2N-4}(f_{(N-1)})
     \\
     &&\ \ \ \ \ \ \ \ \cdots S^2(g_{(2)})S^2(f_{(2)})g_{(1)}f_{(1)},c_{ij}\rangle \\
 &&= \left\langle \sum S^{2N-2}(g_{(N)}f_{(N)})S^{2N-4}(g_{(N-1)}f_{(N-1)})\cdots S^2(g_{(2)}f_{(2)})g_{(1)}f_{(1)},c_{ij} \right\rangle \\
 &&= \left\langle m^{\ast {\rm op}}_N\circ\left( \id\otimes(S^{-1})^{-2}\otimes\cdots\otimes (S^{-1})^{-2N+2}\right)\circ\Delta^\ast_N(gf),c_{ij}\right\rangle \\
 &&= \langle gf,1 \rangle \langle \varepsilon,c_{ij} \rangle
 ~=~ \delta_{ij} \langle f,1 \rangle \langle g,1 \rangle ~=~ \langle f\otimes g,\delta_{ij}(1\otimes 1)\rangle.
\end{eqnarray*}
The proof is now complete since $f$ and $g$ are arbitrary linear functions.
\end{proof}

The following proposition is a conclusion of two lemmas above.

\begin{proposition}\label{8}
Let $H$ be a finite-dimensional Hopf algebra with the dual Chevalley property. Denote $N:=\exp(H_0)<\infty$. Then for any basic multiplicative matrix $\mathcal{C}$ and any $(\mathcal{C},1)$-primitive matrix $\mathcal{X}$, we have $S^{2N}(\mathcal{X})=\mathcal{X}$.
\end{proposition}

\begin{proof}
Denote $\mathcal{C}=(c_{ij})_{r\times r}$ and $\mathcal{X}=(x_1,x_2,\cdots,x_r)^{\rm T}$. Since $\mathcal{C}$ is a multiplicative matrix over $H_0$ as well, Lemma \ref{6}(3) and Lemma \ref{7} imply that for any $1\leq i\leq r$,
\begin{eqnarray*}
 S^{2N}(x_{i})
 &=& \sum\limits_{k_1,k_2,\cdots,k_{2N}=1}^r S\left[c_{k_2k_1} S^2(c_{k_4k_3}) \cdots S^{2N-4}(c_{k_{2N-2}k_{2N-3}})S^{2N-2}(c_{k_{2N}k_{2N-1}})\right] \\
 & & x_{k_1}S^2\left[ c_{k_3k_2} S^2(c_{k_5k_4})\cdots S^{2N-4}(c_{k_{2N-1}k_{2N-2}}) S^{2N-2}(c_{ik_{2N}})\right] \\
 &=& \sum_{k_1=1}^r \delta_{ik_1}S(1)x_{k_1}1 ~=~ x_i.
\end{eqnarray*}
That is to say, $S^{2N}(\mathcal{X})=\mathcal{X}$.
\end{proof}

In Subsection \ref{29}, we have made the convention that a family of coradical orthonormal idempotents $\{e_C\}_{C\in\mathcal{S}}$ is always given. Now recall that according to Proposition \ref{3}(4), the left coideal $H_1{}^1$ could be decomposed as a direct sum:
$$ H_1{}^1=\bigoplus_{C\in\mathcal{S}} {}^C H_1{}^1. $$
On the other hand, if we assume that $\k$ is algebraically closed, then every simple subcoalgebra $C\in\mathcal{S}$ has a basic multiplicative matrix $\mathcal{C}$, and thus each element in ${}^C H_1{}^1$ is a sum of some entries in $(\mathcal{C},1)$-primitive matrices and some elements in $H_0$. This is followed from Lemma \ref{4}. As a consequence, we could obtain the following corollary that the transformation $S^{2N}$ on $H_1{}^1$ equals to identity in this case.

\begin{corollary}\label{30}
Let $H$ be a finite-dimensional Hopf algebra with the dual Chevalley property over an algebraically closed field $\k$. Denote $N:=\exp(H_0)<\infty$. Then $S^{2N}\mid_{H_1{}^1}=\id_{H_1{}^1}$.
\end{corollary}

\begin{proof}
This is because $S^{2N}$ keeps any basic multiplicative matrix $\mathcal{C}$ as well as any $(\mathcal{C},1)$-primitive matrix.
\end{proof}

\subsection{Antipode on $H_1$}

Our first goal in this subsection is to give the following proposition.

\begin{proposition}\label{9}
Let $H$ be a Hopf algebra with the dual Chevalley property. Then $H_1=H_1{}^1\cdot H_0$.
\end{proposition}
\begin{proof} It seems that this is known and we give an approach to prove it for safety. At first, using comultiplication and multiplication one can show that there is a right $H_0$-Hopf module structure on $H_1/H_0$. Secondly, it is straightforward to show that the space of coinvariants of this right Hopf module is exactly $(H_1{}^1+H_0)/H_0$. At last, we can apply the fundamental theorem of Hopf modules (\cite[Propostion 1]{L-S 1969}) to get the result.
\end{proof}

Proposition \ref{9} provides that $H_1{}^1$ and $H_0$ generate $H_1$ by multiplication. Then we can continue to investigate whether $S^{2N}$ could be identified with the identity map on $H_1$ or not.

\begin{proposition}\label{10}
Let $H$ be a finite-dimensional Hopf algebra with the dual Chevalley property over an algebraically closed field $\k$. Denote $N:=\exp(H_0)<\infty$. Then $S^{2N}\mid_{H_1}=\id_{H_1}$.
\end{proposition}

\begin{proof}
We already know that the algebra morphism $S^{2N}$ restricted to subspaces $H_0$ or $H_1{}^1$ is supposed to be the identity (Corollaries \ref{31} and \ref{30}). Now that Proposition \ref{9} gives $H_1=H_1{}^1\cdot H_0$, consequently $S^{2N}\mid_{H_1}=\id_{H_1}$ holds in this case.
\end{proof}

\subsection{Proof of Theorem \ref{2}}

We remark firstly a classical result on coradical filtrations in \cite[Proposition 4]{T-W 1974}, which holds for non-pointed coalgebras as well.

\begin{lemma}\emph{(}\cite[Proposition 4]{T-W 1974}\emph{)}\label{16}
Let $H$ be an arbitrary coalgebra. Let $i$ be a positive integer and $\varphi:H\rightarrow H$ be a coalgebra endomorphism, such that
$(\varphi-\id)(H_j)\subseteq H_{j-1}$ for all $0\leq j\leq i$. Then $(\varphi-\id)(H_{i+1})\subseteq H_i$.
\end{lemma}

\textbf{Proof of Theorem \ref{2} when $\k$ is algebraically closed.}  Combining Lemma \ref{16} and Proposition \ref{10}, it is clear that the statement of Theorem \ref{2} holds when the base field $\k$ is algebraically closed.\qed

Next we want to prove Theorem \ref{2} using the method of field extensions. For the purpose, it is necessary to show that $\exp(H_0)$ and ${\rm Lw}(H)$ are invariant under field extensions due to the dual Chevalley property. The following lemma seems also known.

\begin{lemma}\label{32}
Let $H$ be a finite-dimensional Hopf algebra with the dual Chevalley property over $\k$. Suppose $K$ is a field extension of $\k$, and $H\otimes K$ denotes the extended finite-dimensional $K$-Hopf algebra. Then
\begin{itemize}
  \item[(1)] $H\otimes K$ has the dual Chevalley property with coradical $H_0\otimes K$;
  \item[(2)] The coradical filtration of $H\otimes K$ is $\{H_n\otimes K\}_{n\geq 0}$;
  \item[(3)] Moreover $\exp(H_0\otimes K)=\exp(H_0)$ and ${\rm Lw}(H\otimes K)={\rm Lw}(H)$.
\end{itemize}
\end{lemma}

\begin{proof}
The definition of $H\otimes K$ could be found in \cite[Exercise 7.1.8]{Radford 2012} for example.
\begin{itemize}
  \item[(1)] Regard $H_0\otimes K\hookrightarrow H\otimes K$ as a subspace canonically, which is in fact a Hopf subalgebra over $K$. Meanwhile, $H_0\otimes K$ is cosemisimple because $H_0$ is (\cite[Lemma 1.3]{Larson 1971}). Thus $H_0\otimes K$ is contained in the coradical $(H\otimes K)_0$.

      On the other hand, the coalgebra structure of $H\otimes K$ follows naturally that $\{\wedge^n(H_0\otimes K)\}_{n\geq 0}$ is a coalgebra filtration of $H\otimes K$, which implies that $H_0\otimes K\supseteq (H\otimes K)_0$ by \cite[Proposition 11.1.1]{Sweedler 1969}.

      As a conclusion, the coradical $(H\otimes K)_0=H_0\otimes K$, and the dual Chevalley property for $H\otimes K$ could be obtained since $H_0\otimes K$ is closed under the multiplication evidently.
  \item[(2)] This could be inferred with $(H\otimes K)_0=H_0\otimes K$ as well as
        $$(H_0\otimes K)\wedge (H_n\otimes K)=(H_0\wedge H_n)\otimes K$$
        for each $n\geq 0$.
  \item[(3)] The equation ${\rm Lw}(H\otimes K)={\rm Lw}(H)$ follows immediately from (2). The exponent is invariant under field extensions is stated in \cite[Proposition 2.2(8)]{E-G 1999}.
\end{itemize}
\end{proof}

\textbf{Proof of Theorem \ref{2} for general $\k$.}
Let $K:=\overline{\k}$ be the algebraically closure of $\k$. Then $H\otimes K$ is a finite-dimensional Hopf algebra over the algebraically closed field $K$. We know from Lemma \ref{32} that $H\otimes K$ has the dual Chevalley property, $\exp(H_0\otimes K)=N$ and ${\rm Lw}(H\otimes K)=L$ according to our notations.

Therefore, if we denote the antipode and identity map of $H\otimes K$ by $S_{H\otimes K}$ and $\id_{H\otimes K}$ respectively, then we already know that
\begin{equation}\label{33}
(S_{H\otimes K}{}^{2N}-\id_{H\otimes K})^{L-1}=0
\end{equation}
holds in $\End_K(H\otimes K)$. Now we apply Equation \eqref{33} on element $h\otimes 1_K\in H\otimes K$ for $h\in H$, and obtain
$$(S^{2N}-\id)^{L-1}(h)\otimes 1_K=0.$$
This implies that $(S^{2N}-\id)^{L-1}=0$ holds in $\End_\k(H)$ too.\qed

\section{Two Applications}\label{d}
In this section, we want to give two applications of our main result. Both of them are concern with an important gauge invariant which is called quasi-exponent.
\subsection{A Characterization of Quasi-Exponent}
In \cite{E-G 2002}, they introduced a gauge invariant called the quasi-exponent for finite-dimensional Hopf algebras, which has similar properties with the exponent but always finite. Precisely, the \textit{quasi-exponent} of a finite-dimensional Hopf algebra $H$, denoted by $\qexp(H)$, is defined to be the least positive integer $n$ such that the $n$th power of the Drinfeld element in $D(H)$ is unipotent (\cite[Definition 2.1]{E-G 2002}).

When $H$ is moreover pointed over $\mathbb C$, there is a description of $\qexp(H)$ in the following Etingof-Gelaki's theorem (see \cite[Theorem 4.6]{E-G 2002}).

\begin{lemma}\label{34}
Let $H$ be a finite-dimensional pointed Hopf algebra over $\mathbb C$. Then $\qexp(H)=\exp(G(H))$.
\end{lemma}

Its proof is based on following lemma which is a combination of \cite[Lemma 4.2]{E-G 2002} and \cite[Proposition 4.3]{E-G 2002}.

\begin{lemma}\label{35}
Let $H$ be a finite-dimensional Hopf algebra over $\mathbb C$.
\begin{itemize}
  \item[(1)] If $H$ is filtered and let ${\gr}H$ be its associated graded Hopf algebra. Then $\qexp(H)=\qexp({\gr}H)$.
  \item[(2)] Assume that $H$ is graded with zero part $H_{(0)}$. Then $$\qexp(H)=\lcm(\qexp(H_{(0)}),\ord(S^2)),$$ where $\lcm$ denotes the least common multiple.
\end{itemize}
\end{lemma}

With the help of Lemma \ref{35} and Corollary \ref{11}(1), we could describe the quasi-exponent of finite-dimensional Hopf algebras with the dual Chevalley property.

\begin{proposition}\label{15}
Let $H$ be a finite-dimensional Hopf algebra with the dual Chevalley property over $\mathbb C$. Then
$\qexp(H)=\exp(H_0)$.
\end{proposition}

\begin{proof}
As mentioned in Lemma \ref{14}, the dual Chevalley property implies that $H$ is a filtered Hopf algebra with the filtration $\{H_n\}_{n\geq 0}$. Thus
$$\qexp(H)=\qexp({\gr}H)$$
holds by Lemma \ref{35}(1). Meanwhile, Lemma \ref{35}(2) provides the equation
$$\qexp({\gr}H)=\lcm(\qexp(({\gr}H)_{(0)}),\ord(S_{{\gr}H}{}^2)).$$
As a conclusion, we have
\begin{eqnarray*}
\qexp(H)
 &=& \qexp({\gr}H)
 ~=~ \lcm(\qexp(({\gr}H)_{(0)}),\ord(S_{{\gr}H}{}^2)) \\
 &=& \lcm(\qexp(H_0),\ord(S^2))
 ~=~ \qexp(H_0),
\end{eqnarray*}
as long as we note that the zero part $({\gr}H)_{(0)}=H_0$ and $\ord(S_{{\gr}H}{}^2)=\ord(S^2)$ in our situation. Besides, the last equation holds because of Corollary \ref{11}(1) by noting that $\qexp(H_0)=\exp(H_0)$.
\end{proof}

Since the quasi-exponent is a gauge invariant, we have the following corollary which was known for pointed Hopf algebras (see \cite[Corollary 4.8]{E-G 2002}):

\begin{corollary}
Let $H$ and $H^\prime$ be two finite-dimensional Hopf algebras with the dual Chevalley property over $\mathbb C$. If they are twist equivalent, then $\exp(H_0)=\exp(H^\prime{}_0)$.
\end{corollary}

Note that the quasi-exponent is invariant under taking duals of finite-dimensional Hopf algebras. Thus a dual version of Proposition \ref{15} could be given, which holds for Hopf algebras with the Chevalley property.
Recall that a finite-dimensional Hopf algebra $H$ is said to have the \textit{Chevalley property}, if the tensor product of any two simple $H$-modules is semisimple, or, equivalently, if the radical of $H$ is a Hopf ideal. And it is clear that $H$ has the Chevalley property if and only if $H^\ast$ has the dual Chevalley property.

\begin{corollary}
Let $H$ be a finite-dimensional Hopf algebra with the Chevalley property over $\mathbb C$, and let $H/{\rm Rad}(H)$ be its semisimple quotient. Then
$\qexp(H)=\exp(H/{\rm Rad}(H))$.
\end{corollary}

This corollary implies \cite[Propostion 4.13]{E-G 2002} for the case when $H$ is moreover pointed.

\subsection{Quasi-Exponent of a Pivotal Hopf Algebra}

In this subsection, we concentrate on a kind of semidirect product of a Hopf algebra $H$, which is denoted by $H\rtimes \k\langle S^2\rangle$. It is a pivotal Hopf algebra containing $H$ and appears in some researches such as \cite{Shimizu 2015}. Thus we think that it is interesting to investigate the exponent and quasi-exponent of $H\rtimes \k\langle S^2\rangle$. Let us begin by recalling the corresponding concepts.

\begin{definition}
A Hopf algebra is called \textit{pivotal} if there exists a grouplike element $g\in H$ such that
$$S^2(h)=ghg^{-1}, \;\;\;\;h\in H.$$
Such an grouplike element $g$ is called a \emph{pivotal element} of $H$.
\end{definition}

When a Hopf algebra $H$ is finite-dimensional, the subgroup generated by $S^2\in\End_\k(H)$, which is denoted by $\langle S^2\rangle$, is finite by \cite[Theorem 1]{Radford 1976}.

\begin{definition}
Let $H$ be a finite-dimensional Hopf algebra. The semidirect product \emph{(}or, smash product\emph{)} Hopf algebra $H\rtimes \k\langle S^2\rangle$ of $H$ with $\langle S^2\rangle$ is defined through:
\begin{itemize}
  \item $H\rtimes \k\langle S^2\rangle=H\otimes \k\langle S^2\rangle$ as a coalgebra;
  \item The multiplication is that $(h\rtimes S^{2i})(k\rtimes S^{2j}):=hS^{2i}(k)\rtimes S^{2(i+j)}$ for all $h,k\in H$ and $i,j\in\mathbb Z$;
  \item The unit element is $1\rtimes \id$;
  \item The antipode is then $S_{H\rtimes \k\langle S^2\rangle}:h\rtimes S^{2i}\mapsto S^{-2i+1}(h)\rtimes S^{-2i}$.
\end{itemize}
\end{definition}

\begin{remark}
\emph{1) The algebra
$H\rtimes \k\langle S^2\rangle$ is in fact a Hopf algebra (e.g. \cite[Theorem 2.13]{Molnar 1977}) with $H\cong H\rtimes \id$ as its Hopf subalgebra.}

\emph{2) It is pivotal with a pivotal element $1\rtimes S^2$, since
\begin{eqnarray*}
S_{H\rtimes \k\langle S^2\rangle}^2(h\rtimes S^{2i})
&=& S^2(h)\rtimes S^{2i} \\
&=& (1\rtimes S^2)(h\rtimes S^{2i})(1\rtimes S^2)^{-1}
\end{eqnarray*}
for  $h\in H$ and $i\in\mathbb Z$.}
\end{remark}

For the remaining of this subsection, we aim to establish a formula for the exponent of $H\rtimes \k\langle S^2\rangle$, and then deduce its quasi-exponent when $H$ has the dual Chevalley property. The process goes forward mainly by direct calculations, and the following notation should be given for our purpose. It could be regarded as a special case of twisted exponents introduced in \cite[Definition 3.1]{S-V 2017} and \cite[Definition 3.1]{M-V-W 2016}.

\begin{notation}\label{22}
Let $H$ be a finite-dimensional Hopf algebra. For any $i\in\mathbb Z$, we denote
$$\exp_{2i}(H):=\min\{n\geq 1\mid m_n\circ(\id\otimes S^{2i}\otimes\cdots\otimes S^{2(n-1)i})\circ\Delta_n=u\circ\varepsilon\}.$$
\end{notation}

\begin{remark}\label{r4.10}
\emph{With this notation,  $\exp(H)$ is exactly $\exp_{-2}(H)$ here. We also remark that the research of ``exponents'' firstly begun in \cite{Kashina 1999} and \cite{Kashina 2000} as $\exp_0(H)$ here, which was later studied in \cite{L-M-S 2006} and so on.}
\end{remark}

For simplicity, we always make following conventions:
\begin{itemize}
  \item $\min \varnothing = \infty$;
  \item Any positive integer divides $\infty$;
  \item Any positive integer divided by $\infty$ is also $\infty$.
\end{itemize}
Then whenever finite and infinite,
\begin{equation}\label{36}
\exp_{2i}(H)=\exp(H) ~\text{for all}~ i\in{\mathbb Z}
\end{equation}
when $H$ is involutory. In fact, Equation \eqref{36} holds as long as $H$ is pivotal, which is directly followed from the lemma below:

\begin{lemma}\emph{(}\cite[Lemma 4.2]{Shimizu 2015}\emph{)}\label{20}
Let $H$ be a Hopf algebra and $g\in H$ is grouplike. Denote $\varphi$ as the inner automorphism on $H$ determined by $g$. Then
$$(hg)^{[n]}=\sum h_{(1)}\varphi(h_{(2)})\cdots\varphi^{n-1}(h_{(n)})g^n$$
holds for each $n\geq 1$ and all $h\in H$.
\end{lemma}

Notation \ref{22} helps us to describe the exponent of the semidirect product $H\rtimes \k\langle S^2\rangle$.

\begin{proposition}\label{25}
Let $H$ be a finite-dimensional Hopf algebra. Then
$$\exp(H\rtimes \k\langle S^2\rangle)=\lcm(\exp_{2i}(H)\mid i\in\mathbb Z).$$
\end{proposition}

\begin{proof}
Note that we have $\exp(H\rtimes \k\langle S^2\rangle)=\exp_0(H\rtimes \k\langle S^2\rangle)$ by Equation \eqref{36}. Now for any positive integer $n$ and each $h\in H$, $i\in\mathbb Z$, we calculate that
$$(h\rtimes S^{2i})^{[n]}=\sum h_{(1)}S^{2i}(h_{(2)})\cdots S^{2(n-1)i}(h_{(n)})\rtimes S^{2ni}.$$
Therefore, the $n$th Sweedler power $[n]_{H\rtimes \k\langle S^2\rangle}$ on $H\rtimes \k\langle S^2\rangle$ is trivial if and only if $S^{2ni}=\id$ and
$$m_n\circ(\id\otimes S^{2i}\otimes\cdots\otimes S^{2(n-1)i})\circ\Delta_n=u\circ\varepsilon$$
both hold for all $i\in\mathbb Z$. In other words,
$$
[n]_{H\rtimes \k\langle S^2\rangle} ~\text{is trivial} ~\Longleftrightarrow~
\lcm\left(\ord(S^{2i}),\exp_{2i}(H)\mid i\in\mathbb Z\right)\mid n.
$$
However, we know that $\ord(S^2)\mid \exp_{-2}(H)$ by Corollary \ref{31}. As a conclusion, $\exp(H\rtimes \k\langle S^2\rangle)=\lcm(\exp_{2i}(H)\mid i\in\mathbb Z)$ is obtained.
\end{proof}

We end up by describing the quasi-exponent of $H\rtimes \mathbb C\langle S^2\rangle$ when $H$ has the dual Chevalley property over $\mathbb C$. In this case $H\rtimes \mathbb C\langle S^2\rangle$ also has the dual Chevalley property with the coradical $H_0\rtimes \mathbb C\langle S^2\rangle$.

\begin{proposition}
Let $H$ be a finite-dimensional Hopf algebra with the dual Chevalley property over $\mathbb C$. Then
$$\qexp(H\rtimes \mathbb C\langle S^2\rangle)=\exp(H_0)=\qexp(H).$$
\end{proposition}
\begin{proof} By Proposition \ref{15}, we know that $\exp(H_0)=\qexp(H)$, and $\qexp(H\rtimes \mathbb C\langle S^2\rangle)=\exp(H_0\rtimes \mathbb C\langle S^2\rangle)$ which equals to $\lcm(\exp_{2i}(H_0)\mid i\in\mathbb Z)$ by Proposition \ref{25}. Since $H_0$ is semisimple over $\mathbb C$, $(S\mid_{H_0})^2=\id_{H_0}$. This implies that
 $\lcm(\exp_{2i}(H_0)\mid i\in\mathbb Z)=\exp(H_0)$.
\end{proof}

%
%

\end{document}